\PassOptionsToPackage{final}{changes}
\documentclass{birkjour}
\usepackage[dvipsnames]{xcolor}
\usepackage{changes}
\setaddedmarkup{{\color{Blue}#1}}
\setdeletedmarkup{{\color{Orange}#1}}

\usepackage{amssymb}
\usepackage{amsmath}
\usepackage{amsthm}
\usepackage{color}
\usepackage{enumerate}
\usepackage{url}

\usepackage{natbib}
\usepackage[linktocpage=true]{hyperref}
\usepackage{cleveref}
\hypersetup{citecolor=blue, linkcolor=blue, colorlinks=true}
\usepackage{blindtext}
\usepackage{stmaryrd}
\usepackage{doi}
\usepackage{jabbrv}
\usepackage[skins,breakable]{tcolorbox}

\theoremstyle{plain}
\newtheorem{theorem}{Theorem}[section]
\newtheorem*{theorem*}{Theorem}
\newtheorem{lemma}[theorem]{Lemma}
\newtheorem{corollary}[theorem]{Corollary}

\theoremstyle{definition}

\newtheorem*{rexample}{Example}
\tcolorboxenvironment{rexample}{
    blanker,
    breakable,
    left=5mm,
    before skip=10pt,
    after skip=10pt,
    borderline west={5pt}{0pt}{gray!65},
    pad before break=2pt,
    pad after break=2pt,
    use color stack,
}

\setcitestyle{numbers,square}

\usepackage{mathtools}

\newcommand\definedTerm\textit

\DeclarePairedDelimiter\abs{\lvert}{\rvert}

\newcommand\Alg{\mathrm{Alg}}
\newcommand\algA{\mathcal A}
\newcommand\algB{\mathcal B}

\DeclareMathOperator\Aut{Aut}
\newcommand\bvecB B

\DeclareMathOperator\Cl{Cl}

\DeclareMathOperator\Der{Der}
\newcommand\derD{D}
\DeclareMathOperator\E{E}
\renewcommand\emptyset\varnothing

\DeclarePairedDelimiter\eqc{[}{]}
\renewcommand\epsilon\varepsilon
\DeclareMathOperator\Ex{\Lambda}

\newcommand\ezero{{e_0}}
\newcommand\F{\mathbb F}
\DeclarePairedDelimiter\gen{\langle}{\rangle}
\renewcommand\geq\geqslant

\newcommand\gradeInv{\alpha}

\DeclareMathOperator\Grp{\mathsf{Grp}}

\DeclareMathOperator\id{id}

\DeclareMathOperator\im{im}

\newcommand\isomI i
\newcommand\isomF f
\newcommand\isomG g
\newcommand\isomP p
\newcommand\isomS s

\renewcommand\leq\leqslant
\DeclareMathOperator\LieAlg{\mathsf{LieAlg}}

\newcommand\mvecA A
\newcommand\mvecB B
\newcommand\mvecC C
\newcommand\mvecD D
\newcommand\mvecX X
\newcommand\mvecY Y
\newcommand\mvecZ Z

\let\O\undefined
\DeclareMathOperator\O{O}
\let\o\undefined
\DeclareMathOperator\o{\mathfrak o}
\DeclareMathOperator\operp{\mathbin{\raisebox{1pt}{\(\scriptstyle\perp\mkern-16mu\bigcirc\)}}}
\DeclarePairedDelimiter\paren{(}{)}

\renewcommand\phi\varphi

\newcommand\pss I

\newcommand\ptP P
\newcommand\ptQ Q
\newcommand\ptR R
\newcommand\qfQ q

\DeclareMathOperator\Quad{\mathsf{QVec}}
\newcommand\R{\mathbb R}
\DeclareMathOperator\rad{rad}

\newcommand\sbfB b

\DeclareMathOperator\se{\mathfrak{se}}
\DeclarePairedDelimiter\set{\{}{\}}

\DeclareMathOperator\so{\mathfrak{so}}
\DeclareMathOperator\spann{span}

\renewcommand\trianglelefteq\trianglelefteqslant
\newcommand\vecL\ell
\newcommand\vecA a
\newcommand\vecB b
\newcommand\vecM m
\newcommand\vecU u
\newcommand\vecV v
\newcommand\vecW w
\newcommand\vecX x
\newcommand\vecY y

\newcommand\vsM M
\newcommand\vsU U
\newcommand\vsV V
\newcommand\vsW W

\apptocmd{\thebibliography}{\raggedright}{}{}

\begin{document}
\title[Exploiting degeneracy in PGA]{Exploiting degeneracy in projective\\ geometric algebra}

\author[J. Bamberg]{John Bamberg}
\address{Centre for the Mathematics of Symmetry and Computation, Department of Mathematics and Statistics, The University of Western Australia, 35 Stirling Highway, Crawley, W. A. 6019, Australia.}
\email{John.Bamberg@uwa.edu.au}

\author[J. Saunders]{Jeff Saunders}
\address{Department of Mathematics and Statistics, The University of Western Australia, 35 Stirling Highway, Crawley, W. A. 6019, Australia.}
\email{Jeff.Saunders@uwa.edu.au}


\begin{abstract}
The last two decades, since the seminal work of \citet{selig1}, has seen projective geometric algebra (PGA) gain popularity as a modern coordinate-free framework for doing classical Euclidean geometry and other Cayley-Klein geometries. This framework is based upon a degenerate Clifford algebra, and it is the purpose of this paper to delve deeper into its internal algebraic structure and extract meaningful information for the purposes of PGA. This includes exploiting the split extension structure to realise the natural decomposition of elements of this Clifford algebra into Euclidean and ideal parts. This leads to a beautiful demonstration of how Playfair's axiom for affine geometry arises from the ambient degenerate quadratic space. The highlighted split extension property of the Clifford algebra also corresponds to a splitting of the group of units and the Lie algebra of bivectors. Central to these results is that the degenerate Clifford algebra \(\Cl(\vsV)\) is isomorphic to the twisted trivial extension \(\Cl(\vsV/\F\ezero)\ltimes_\gradeInv\Cl(\vsV/\F\ezero)\), where \(\ezero\) is a degenerate vector and $\gradeInv$ is the grade-involution.
\end{abstract}

\maketitle

\section{Introduction}
Euclidean projective (or plane-based) geometric algebra (PGA) is a new and promising approach to Euclidean geometry, introduced by \citet{selig1,selig2} and developed further by \citet{gunn1,gunn2}. Euclidean PGA uses a degenerate \((n+1)\)-dimensional quadratic space, whose elements represent the hyperplanes of \(n\)-dimensional Euclidean space; its weak orthogonal group is then isomorphic to the group of motions of Euclidean space (see \citet{havlicek1,havlicek2}). PGA also has elliptic and hyperbolic variants, which use \emph{nondegenerate} \((n+1)\)-dimensional quadratic spaces, and whose Clifford algebras enjoy a natural polarity as a result of that nondegeneracy. See \citet{gunn3} for in-depth coverage of both the uses of this polarity and some alternative dualities available in Euclidean PGA. A theme across discussions of duality has been the idea that Euclidean PGA's lack of polarity is a deficiency to be either worked around with such alternatives, or avoided by working in the Clifford algebra of a higher-dimensional nondegenerate quadratic space. In this article we point to some additional internal structure, found only in degenerate Clifford algebras, that may lead to new perspectives.

\Citet{dorst_de_keninck2} make heavy use of the notion of splitting elements of the Clifford algebra into \definedTerm{Euclidean} and \definedTerm{ideal} parts\added{; \citet{lengyel1} similarly decomposes elements into \definedTerm{bulk} and \definedTerm{weight} parts}. We will see in \Cref{section:quadratic_space} that this splitting comes from the fact that the degenerate quadratic space of planes in Euclidean space is a split extension of the space of parallel classes of planes, and that the latter is isometric to the cotangent space at any point. This can be used to demonstrate Playfair's axiom (better known as Euclid's fifth postulate), giving our first hint at the deep links between degenerate quadratic spaces and metric affine geometry. We also get, via functoriality, resulting split extensions of the corresponding Clifford algebra (\Cref{section:clifford_algebra}), as well as its group of units and the Lie algebra of its bivectors (\Cref{section:group_of_units}). In particular we have the beautiful coincidence that all of the ideal parts of elements of the Clifford algebra live in an ideal of the algebra.

The mathematical treatment in this paper is entirely coordinate-free and works for any field not of characteristic 2, but in the interests of applicability and accessibility, the running example uses the real numbers and a basis. Some readers may be surprised to see the semidirect product symbol \(\ltimes\), rather than its mirror image \(\rtimes\) which may be more commonly seen in group theory texts; for compatibility and consistency with the Clifford algebra convention of acting on the left, the authors have chosen the former notation.

\section{Decomposing a degenerate quadratic space}\label{section:quadratic_space}
For the purposes of this article, we work over a fixed field \(\F\) not of characteristic 2, and assume algebras are associative unless otherwise specified. A \replaced{\definedTerm{symmetric bilinear form}}{symmetric bilinear form} \(\sbfB\) on a vector space \(\vsV\) is a bilinear map \(\sbfB:\vsV\times\vsV\to\F\) such that for all \(\vecU,\vecV\in\vsV\), we have that \(\sbfB(\vecU,\vecV)=\sbfB(\vecV,\vecU)\). Note that we place no further restrictions, such as definiteness, on \(\sbfB\). A \definedTerm{quadratic space} (also known as a \definedTerm{metric vector space} or \definedTerm{symmetric bilinear space}) \(\vsV\coloneqq(\vsV,\sbfB)\) is a vector space \(\vsV\) together with a symmetric bilinear form \(\sbfB\) on \(\vsV\)\footnote{Many authors define a quadratic space in terms of a quadratic form rather than a symmetric bilinear form; the notions are equivalent in characteristic not equal to 2, and symmetric bilinear forms will be more convenient for us to work with.}.

\begin{rexample}
    We will use as a running example the real projective space of planes in three-dimensional Euclidean space; all later examples refer back to this. We begin with the four-dimensional vector space \(\vsV\coloneqq\spann\set{e_0,e_1,e_2,e_3}\) over \(\R\) and say that the element \(\vecV_0e_0+\vecV_1e_1+\vecV_2e_2+\vecV_3e_3\in\vsV\) represents the plane with equation \(\vecV_0+\vecV_1x+\vecV_2y+\vecV_3z=0\). We immediately note a few things: first, that \(\vecV\) and \(\lambda\vecV\) represent the same plane for any nonzero scalar \(\lambda\in\R\); this is a homogeneous representation (thus zero will not represent a plane). Second, elements \(\vecU,\vecV\in\vsV\) represent parallel planes if and only if \(\vecU=\lambda\vecV+\mu e_0\) for some \(\lambda,\mu\in\R\) with \(\lambda\neq0\). Multiples of \(e_0\) itself do not represent planes, but the elements of the sequence \(m\mapsto(1/m)\vecV+e_0\) represent a sequence of parallel planes regressing monotonically from a point.

    The two-parameter bundle of planes through a given point \(\ptP\) forms a three-dimensional (projective dimension two) subspace \(\vsV_\ptP\), the one-parameter bundle of planes through a line forms a two-dimensional subspace, and each plane alone forms a one-dimensional subspace. We will also be using the three-dimensional quotient space \(\vsV/\R e_0\), whose elements represent parallel classes of planes; the canonical projection \(\pi\) sends a representative of a plane to a representative of that plane's parallel class. This is also a homogeneous representation, so the cosets \(\eqc\vecV\coloneqq\vecV+\R e_0\) and \(\lambda\eqc\vecV=\eqc{\lambda\vecV}\) represent the same parallel class of planes.
    
    We define the map \[\begin{split}
        \sbfB:\vsV\times\vsV&\to\R\\
        (\vecU,\vecV)&\mapsto\vecU_1\vecV_1+\vecU_2\vecV_2+\vecU_3\vecV_3,
    \end{split}\] where \(\vecU=\vecU_0e_0+\vecU_1e_1+\vecU_2e_2+\vecU_3e_3\) and \(\vecV=\vecV_0e_0+\vecV_1e_1+\vecV_2e_2+\vecV_3e_3\). It is straightforward to confirm that \(\sbfB\) is bilinear and symmetric, giving us the quadratic space \((\vsV,\sbfB)\). This lets us find the \definedTerm{magnitude} of an element \(\vecV\) of \(\vsV\) with \(\abs\vecV^2\coloneqq\sbfB(\vecV,\vecV)\geq0\), giving two canonical (\definedTerm{unit-magnitude}) representatives \(\vecV/\abs\vecV\) and \(-\vecV/\abs\vecV\) for each plane. We will treat these two representatives as having opposite orientation (facing in opposite directions), and say that \(\vecV\) and \(\lambda\vecV\) have the same orientation for any \emph{positive} scalar \(\lambda\in\R\). Note that \begin{equation}\label{equation:cosine_law}
        \sbfB(\vecU,\vecV)=\abs\vecU\abs\vecV\cos\theta,
    \end{equation} where \(\theta\) is the dihedral angle between the planes represented by \(\vecU\) and \(\vecV\) (taking orientation into account).
\end{rexample}

An element \(\vecV\in\vsV\) is \definedTerm{null} (or \definedTerm{isotropic}) if \(\sbfB(\vecV,\vecV)=0\), and \definedTerm{nonnull} (or \definedTerm{anis\added{o}tropic}) otherwise. More strongly, if \(\sbfB(\vecV,\vecV')=0\) for all \(\vecV'\in\vsV\), then \(\vecV\) is called \definedTerm{degenerate} (or \definedTerm{singular}); otherwise it is \definedTerm{nondegenerate} (or \definedTerm{regular}). The subspace of all degenerate elements in \(\vsV\) is called its \definedTerm{radical} and denoted \(\rad\vsV\). A quadratic space with nontrivial radical is called \definedTerm{degenerate}. We will call a subspace \(\vsU\leq\rad\vsV\) a \definedTerm{radical subspace} of \(\vsV\); we say that \(\vsU\) is \definedTerm{radical} in \(\vsV\) and write \(\vsU\trianglelefteq\vsV\).

\begin{rexample}
    For all \(\vecV\in\vsV\) we have that \(\sbfB(e_0,\vecV)=0\), so \(\R e_0\trianglelefteq\vsV\) and \(\vsV\) is degenerate. In fact \(\rad\vsV=\R e_0\) as no degenerate elements lie outside \(\R e_0\). There are also no null nondegenerate elements in \(\vsV\). Let \(\vecV\in\vsV\); then \(\vecV+e_0\) is the representative of a parallel plane with the same orientation and magnitude. As \Cref{equation:cosine_law} suggests, for \(\vecW\in\vsV\) representing any plane, \[
        \sbfB(\vecV+e_0,\vecW)=\sbfB(\vecV,\vecW)+\sbfB(e_0,\vecW)=\sbfB(\vecV,\vecW)+0=\sbfB(\vecV,\vecW),
    \] so \(\sbfB\) is well-defined on the quotient space \(\vsV/\R e_0\).
\end{rexample}

A \definedTerm{linear map} from a quadratic space \((\vsV,\sbfB)\) to a quadratic space \((\vsV',\sbfB')\) is just a linear map from \(\vsV\) to \(\vsV'\). A (not necessarily bijective) linear map \(\isomF:(\vsV,\sbfB)\to(\vsV',\sbfB')\) is an \definedTerm{isometry} if \(\sbfB'(\isomF(\vecU),\isomF(\vecV))=\sbfB(\vecU,\vecV)\) for all vectors \(\vecU,\vecV\) in \(\vsV\). Given two quadratic spaces \(\vsU\) and \(\vsV\), we say that \(\vsU\) is \definedTerm{isometric} to \(\vsV\) if there exists a bijective isometry \(\isomF:\vsU\to\vsV\), and write \(\vsU\simeq\vsV\). The following lemma comes from an exercise in \citet[Exercise 8, p. 144]{snapper_troyer1}: 

\begin{lemma}\label{theorem:quotient_quadratic_space}The symmetric bilinear form \(\sbfB\) is well-defined on the quotient space \(\vsV/\vsU\) if and only if \(\vsU\) is radical in \(\vsV\).\end{lemma}

When the condition of the above lemma holds, then \((\vsV/\vsU,\sbfB)\) is a quadratic space in its own right, and the canonical projection \(\pi:(\vsV,\sbfB)\to(\vsV/\vsU,\sbfB)\) is an isometry. As the identity map and compositions of isometries are isometries, and using standard properties of function composition, we can form the category \(\Quad_\F\), whose objects are quadratic spaces over \(\F\) and whose morphisms are isometries. We will refer to the set of isometries from \(\vsU\) to \(\vsV\) as \(\Quad_\F(\vsU,\vsV)\). A bijective isometry from \(\vsV\) to itself is called an \definedTerm{orthogonal transformation} of \(\vsV\); the set of these forms the \definedTerm{orthogonal group} \(\O(\vsV)\). Following \citet{havlicek2}, we shall also define the \definedTerm{weak orthogonal group} \(\O'(\vsV)\) to consist of those orthogonal transformations that fix \(\rad\vsV\) elementwise. Note that \(\O(\vsV)=\O'(\vsV)\) exactly when \(\vsV\) is nondegenerate.

\begin{rexample}
    Consider an isometry \(\isomF\in\Quad_\R(\vsV,\vsV)\). We must have that \(\isomF\) fixes \(\R e_0\) setwise, as each element \(\lambda e_0\) for some \(\lambda\in\R\) must be mapped to a null element and all null elements are in \(\R e_0\). Similarly \(\isomF\) fixes \(\vsV\setminus\R e_0\) setwise: every element \(\vecV\in\vsV\setminus\R e_0\) is nonnull and must be mapped to a nonnull element, of which all lie outside \(\R e_0\). Geometrically, \(\O(\vsV)\) consists of all (bijective) similarity transformations of space, and \(\O'(\vsV)\)\ of all rigid transformations; hence the latter group is isomorphic to the Euclidean group \(\E(3)\). As the quotient space \(\vsV/\replaced{\R e_0}{\gen e_0}\) is nondegenerate, \(\O(\vsV/\replaced{\R e_0}{\gen e_0})=\O'(\vsV/\replaced{\R e_0}{\gen e_0})\), and in fact they are isomorphic to the more famous orthogonal group \(\O(3)\).
\end{rexample}

Let \(\vsU\), \(\vsV\), and \(\vsW\) be quadratic spaces and \(\isomI:\vsU\to\vsV\) and \(\isomP:\vsV\to\vsW\) isometries such that the sequence \[
    0\to\vsU\xrightarrow\isomI\vsV\xrightarrow\isomP\vsW\to0
\] is exact, i.e., \(\isomI\) is injective, \(\isomP\) is surjective, and \(\im\isomI=\ker\isomP\). Then we say that \(\vsV\) is an \definedTerm{extension} of \(\vsW\) by \(\vsU\). It is easy to see that the kernel of any isometry from \(\vsV\) to \(\vsW\) is a radical subspace, so only degenerate quadratic spaces can be nontrivial extensions (i.e., where \(\vsU\neq0\)).

\begin{rexample}
    The inclusion map \(\iota:\R e_0\to\vsV\) is injective and the canonical projection \(\pi:\vsV\to\vsV/\R e_0\) surjective, and the image of \(\iota\) is exactly the kernel of \(\pi\). That is, all multiples of \(e_0\), and only multiples of \(e_0\), are sent to \(\eqc0\) in the quotient. Thus \(\vsV\) is an extension of \(\vsV/\R e_0\) by \(\R e_0\).
\end{rexample}

We say that an extension \(\vsU\to\vsV\xrightarrow\isomP\vsW\) \definedTerm{splits} if \(\isomP\) has a \definedTerm{section}, i.e., an isometry \(\isomS:\vsW\to\vsV\) such that \(\isomP\circ\isomS=\id_\vsW\). Note that if \(\vsW\leq\vsV\) then the extension splits using \(\isomS=\id_\vsW\), so a split extension is like finding an isometric copy of \(\vsW\) within \(\vsV\). An internal direct sum \(\vsW\oplus\vsU\) in a quadratic space \(\vsV\) is an (internal) \definedTerm{orthogonal sum} if \(\sbfB(\vecU,\vecW)=0\) for all \(\vecU\in\vsU\) and \(\vecW\in\vsW\), and is written \(\vsW\operp\vsU\) if so. A \definedTerm{complement}\footnote{When \(\vsV\) is infinite-dimensional, finding a complement may require the Axiom of Choice.} of a subspace \(\vsU\) of \(\vsV\) is a subspace \(\vsW\) such that the vector space \(\vsV=\vsW\oplus\vsU\). If, moreover, \(\vsV=\vsW\operp\vsU\) then we call \(\vsW\) an \definedTerm{orthogonal complement}. This differs from the standard definition of orthogonal complement; when \(\vsU\) is nondegenerate then it has a unique orthogonal complement which is equal to the set given by the usual definition.

\begin{rexample}
    For each point \(\ptP\) in \replaced{Euclidean space}{the Euclidean plane}, the subspace \(\vsV_\ptP\) of planes through \(\ptP\) is a complement of \(\R e_0\), and in fact these are the only complements of \(\R e_0\).
\end{rexample}

We include the following lemma here as we will be referring back to it multiple times later; the proof is straightforward and omitted.

\begin{lemma}\label{theorem:orthogonal_complements}Let \(\vsU\) and \(\vsW\) be subspaces of \(\vsV\) with \(\vsU\trianglelefteq\vsV\). Then the following are equivalent:\begin{enumerate}
    \item\(\vsW\) is an orthogonal complement of \(\vsU\),
    \item\(\vsW\) is a complement of \(\vsU\), and
    \item\(\vsW\) is isometric to \(\vsV/\vsU\).
\end{enumerate} Moreover, for each complement \(\vsW\), the canonical projection \(\pi:\vsV\to\vsV/\vsU\) has a unique section with image \(\vsW\), given by \(\omega_\vsW\coloneqq(\pi\circ\iota_\vsW)^{-1}\), where \(\iota_\vsW\) is the inclusion of \(\vsW\) in \(\vsV\).\end{lemma}

\begin{corollary}\label{theorem:playfair_decomposition}Suppose \(\ezero\in\rad\vsV\) and \(\vsV=\vsW\operp\F\ezero\). Then for every \(\vecV\in\vsV\) there exists \(\lambda\in\F\) such that \(\vecV=\vecW+\lambda\ezero\), where \(\vecW=(\omega_\vsW\circ\pi)(\vecV)\in\vsW\).\end{corollary}

We refer to \(\omega_\vsW\) as the \definedTerm{canonical section} of \(\pi\) for \(\vsW\). This lemma shows that every quadratic space extension splits; in short, a quadratic space contains all of its quotients. We now know that any quadratic space can be written as the orthogonal sum of any of its radical subspaces and any complement of that radical subspace. This is not the only way to decompose a quadratic space (see for example \citet[\S4]{lam1}), but this approach using a split extension is also used for algebras and groups, and will give us a new perspective on degenerate Clifford algebras.

\begin{rexample}
    For any point \(\ptP\), the subspace \(\vsV_\ptP\) is isometric to \(\vsV/\R e_0\); let us see what that isometry looks like in coordinates. Let \(\ptP=(x,y,z)\) and \(\vecV=\vecV_0e_0+\vecV_1e_1+\vecV_2e_2+\vecV_3e_3\). Then \(\vecV\) belongs to the coset \[
        \eqc\vecV=\{(\vecV_0+\lambda)e_0+\vecV_1e_1+\vecV_2e_2+\vecV_3e_3\mid\lambda\in\R\}.
    \] For a plane represented by an element of \(\eqc\vecV\) to pass through \(\ptP=(x,y,z)\), we need the unique solution \(\lambda=-\vecV_0-\vecV_1x-\vecV_2y-\vecV_3z\). We thus define the map \[\begin{split}
        \omega_{(x,y,z)}:\vsV/\R e_0&\to\vsV_{(x,y,z)}\\
        \eqc{\vecV_0e_0+\vecV_1e_1+\vecV_2e_2+\vecV_3e_3}&\mapsto(\vecV_0+\lambda)e_0+\vecV_1e_1+\vecV_2e_2+\vecV_3e_3,
    \end{split}\] which is the isometry we need between \(\vsV/\R e_0\) and \(\vsV_\ptP\). The existence of this isometry shows that every plane through \(\ptP\) belongs to a unique parallel class of planes, and that every parallel class of planes contains a unique plane through \(\ptP\). We thus arrive at a nice algebraic demonstration of the fundamental notion of affine geometry:
\end{rexample}

\begin{theorem}[Playfair's Axiom in 3D]Let \(\ptP\) be a point and \(\Pi\) a plane in Euclidean space. Then there is a unique plane through \(\ptP\) parallel to \(\Pi\).\end{theorem}
\begin{proof}
    Let \(\vecV\in\vsV\) be a representative of \(\Pi\). The projection \(\pi\) sends \(\vecV\) to the coset \(\eqc\vecV\) representing \(\Pi\)'s parallel class \(\eqc\Pi\); the section \(\omega_\ptP\) picks out the representative (with the same orientation and magnitude as \(\vecV\)) of the unique plane in \(\eqc\Pi\) that passes through \(\ptP\).
\end{proof}

For this reason, we refer to \(\pi_\vsW=\omega_\vsW\circ\pi\) as the \definedTerm{Playfair projection} onto \(\vsW\), and to the decomposition \(\vecV=\vecW+\lambda\ezero\) given in \Cref{theorem:playfair_decomposition} as the \definedTerm{Playfair decomposition} of \(\vecV\) at \(\vsW\). We will return to the Playfair projection repeatedly in the sequel.

\begin{rexample}
    Note that if \(\ezero\notin\rad\vsV\), as is the case in elliptic and hyperbolic PGA, then by \Cref{theorem:quotient_quadratic_space} we cannot form the quotient space \(\vsV/\R\ezero\), and thus Playfair's axiom does not hold. The subspaces \(\vsV_\ptP\) and \(\vsV_\ptQ\) for points \(\ptP\) and \(\ptQ\) still turn out to be isometric, but they are not canonically so without \(\pi\), \(\omega_\ptP\), and \(\omega_\ptQ\). Thus it is precisely the degeneracy of \(\vsV\) that allows us to use it to represent metric affine geometry.
\end{rexample}

\section{Decomposing a degenerate Clifford algebra}\label{section:clifford_algebra}
Let \(\vsV=(\vsV,\sbfB)\) be a quadratic space. The \definedTerm{Clifford algebra} of \(\vsV\) is the quotient algebra \[
    \Cl(\vsV)=T(\vsV)/J(\vsV),
\] where \(T(\vsV)\) is the tensor algebra of \(\vsV\) and \(J(\vsV)\) is the ideal of \(T(\vsV)\) generated by \[
    \vecU\otimes\vecV+\vecV\otimes\vecU-2\sbfB(\vecU,\vecV),
\] for all \(\vecU,\vecV\in\vsV\). As we are working over a field not of characteristic 2, this construction is equivalent to the one given in \citet[Chapter III]{chevalley1} (which uses a quadratic form). The universal property of the Clifford algebra, Chevalley's Theorem 3.1, implies the existence of the functor \(\Cl:\Quad_\F\to\Alg_\F\), where \(\Alg_\F\) denotes the category of associative algebras and associative algebra homomorphisms over \(\F\). This functor maps an isometry \(\isomF:\vsV\to\vsV'\) to its \definedTerm{Clifford extension}, the unique unital algebra homomorphism \(\Cl(\isomF):\Cl(\vsV)\to\Cl(\vsV')\) extending \(\isomF\). In fact, every unital algebra homomorphism \(\phi:\Cl(\vsV)\to\Cl(\vsV')\) such that \(\phi(\vsV)\subseteq\vsV'\) is the Clifford extension of an isometry, so we call \(\phi\) a \definedTerm{Clifford algebra homomorphism}. We also define a \definedTerm{Clifford algebra derivation} \(\derD:\Cl(\vsV)\to\replaced{\Cl(\vsV)}{\Cl(\vsV')}\) as a derivation -- i.e., a linear map satisfying \(\derD(\mvecX\mvecY)=\mvecX\derD(\mvecY)+\derD(\mvecX)\mvecY\) for all \(\mvecX,\mvecY\in\Cl(\vsV)\) -- such that \(\derD(\vsV)\subseteq\vsV\deleted{'}\).

We recall the linear isomorphism \(\theta:\Cl(\vsV)\to\Ex(\vsV)\), as shown by \citet[p.39]{chevalley1}. We define the wedge product on the Clifford algebra via said isomorphism: \[\begin{split}
    \wedge:\Cl(\vsV)\times\Cl(\vsV)&\to\Cl(\vsV)\\
    (\mvecX,\mvecY)&\mapsto\mvecX\wedge\mvecY\coloneqq\theta^{-1}(\theta(\mvecX)\wedge\theta(\mvecY)),
\end{split}\] and the direct sum decomposition \begin{align*}
    \Cl(\vsV)={}&\Cl^0(\vsV)\oplus\Cl^1(\vsV)\oplus\dotsb\oplus\Cl^n(\vsV)\\
    \coloneqq{}&\theta^{-1}(\Ex^0(\vsV))\oplus\theta^{-1}(\Ex^1(\vsV))\oplus\dotsb\oplus\theta^{-1}(\Ex^n(\vsV)).
\end{align*} We refer to \(\Cl^k(\vsV)\) as the \definedTerm{grade-\(k\) part} of \(\Cl(\vsV)\); note that this gives only a vector space grading of \(\Cl(\vsV)\) and not an algebra grading. In any quadratic space, the map \((-\id):\vecV\mapsto-\vecV\) is an orthogonal transformation, whose Clifford extension is the \definedTerm{main automorphism} (or \definedTerm{grade involution}) \(\gradeInv\coloneqq\Cl(-\id)\) and has the effect of negating the odd-graded parts of \(\Cl(\vsV)\) and preserving the even-graded ones. The tensor algebra \(T(\vsV)\) is finitely spanned by simple tensors \(\vecV_1\otimes\dotsb\otimes\vecV_k\), and the Clifford algebra, being a quotient of \(T(\vsV)\), inherits this property of being finitely spanned by finite products of elements of \(\vsV\).

From now on we assume that \(\vsV\) is degenerate and pick a one-dimensional radical subspace \(\F\ezero\). We assume also that \(\F\ezero\) has an orthogonal complement \(\vsW\) in \(\vsV\); recall by \Cref{theorem:orthogonal_complements} that every complement of \(\F\ezero\) is orthogonal. These assumptions hold for our running example, but also for a general degenerate quadratic space \(\vsV\) if it is finite-dimensional or assuming the Axiom of Choice, including when \(\F\ezero\) is a proper subspace of \(\rad\vsV\).

\begin{rexample}The Clifford algebra \(\Cl(\vsV)\) of our running example is known as \definedTerm{Euclidean PGA}. In this context, an element \(\vecV\in\vsV\) represents not just a plane but also reflection in that plane; see \citet{gunn1} or \citet{dorst_de_keninck1} for an introduction.\end{rexample}

\begin{lemma}\label{theorem:commuting_ezero}Let \(\mvecX\in\Cl(\vsV)\). Then \(\ezero\mvecX=\gradeInv(\mvecX)\ezero\).\end{lemma}
\begin{proof}
    Decompose \(\mvecX\) into a linear combination of products of elements of \(\vsV\) and apply the fact that \(\sbfB(\ezero,\vecV)=0\) for all \(\vecV\in\vsV\).
\end{proof}

Recall from \Cref{theorem:playfair_decomposition} that for every \(\vecV\in\vsV\) there exists \(\lambda\in\F\) such that \(\vecV=\pi_\vsW(\vecV)+\lambda\ezero\), where we call \(\pi_\vsW\) the Playfair projection onto a complement \(\vsW\) of \(\F\ezero\). We now extend this result to the Clifford algebra.

\begin{lemma}\label{theorem:clifford_algebra_playfair_decomposition}For every \(\mvecX\in\Cl(\vsV)\) there exists \(\mvecY\in\Cl(\vsW)\) such that \(\mvecX=\Cl(\pi_\vsW)(\mvecX)+\mvecY\ezero\).\end{lemma}
\begin{proof}
    It suffices to decompose a product \(\vecV_1\vecV_2\dotsm\vecV_k\) of elements of \(\vsV\). By \Cref{theorem:playfair_decomposition}, we have that \[
        \vecV_1\vecV_2\dotsm\vecV_k=(\vecW_1+\lambda_1\ezero)(\vecW_2+\lambda_2\ezero)\dotsm(\vecV_k+\lambda_k\ezero).
    \] As \(\ezero^2=0\) and using \Cref{theorem:commuting_ezero}, then \[\begin{split}
        \vecV_1\vecV_2\dotsm\vecV_k&=\vecW_1\vecW_2\dotsm\vecW_k+\sum_{i=1}^k\lambda_i\vecW_1\dotsm\vecW_{i-1}\ezero\vecW_{i+1}\dotsm\vecW_k\\
        &=\Cl(\pi_\vsW)(\vecV_1\vecV_2\dotsm\vecV_k)+\paren*{\sum_{i=1}^k\lambda_i\vecW_1\dotsm\vecW_{i-1}\gradeInv(\vecW_{i+1}\dotsm\vecW_k)}\ezero;
    \end{split}\] hence both \(\Cl(\pi_\vsW)(\vecV_1\vecV_2\dotsm\vecV_k)\) and \(\mvecY\coloneqq\sum_{i=1}^k\lambda_i\vecW_1\dotsm\vecW_{i-1}\gradeInv(\vecW_{i+1}\dotsm\vecW_k)\) are in \(\Cl(\vsW)\).
\end{proof}

We refer to the decomposition above as the \definedTerm{Playfair decomposition} of an element of \(\Cl(\vsV)\) at \(\vsW\). Note also that when we have a basis \(B\cup\set\ezero\) for \(\vsV\) with the property that \(B\) is a basis for \(\vsW\), then we can read off the Playfair decomposition of an element \(\mvecX\in\Cl(\vsV)\) by writing \(\mvecX\) in terms of \(B\cup\set\ezero\) and selecting the terms that do or do not contain \(\ezero\); this is common practice among PGA practitioners.

\begin{rexample}The Playfair decomposition of \(\mvecX\) at \(\vsV_\ptP\), or more simply at \(\ptP\) itself, separates \(\mvecX\) into the parts \(\Cl(\pi_\ptP)(\mvecX)\) and \(\mvecY\ezero=\mvecX-\Cl(\pi_\ptP)(\mvecX)\). The former part is incident with \(\ptP\) in the sense that \(\Cl(\pi_\ptP)(\mvecX)\wedge\ptP=0\), so we refer to it as the part of \(\mvecX\) \definedTerm{at \(\ptP\)}. The other part of \(\mvecX\) is at infinity in the sense that it lies within \(\Cl(\vsV)\ezero\), so we call it the part of \(\mvecX\) \definedTerm{at infinity} with respect to \(\ptP\). Note that the parts of \(\mvecX\) at infinity with respect to the points \(\ptP\) and \(\ptQ\) may differ when \(\ptP\neq\ptQ\), while the parts at \(\ptP\) and \(\ptQ\) are in some sense the same; this is due to the fact that \(\pi_\ptP\) and \(\pi_\ptQ\) factor through \(\pi\). The Playfair decomposition is nearly identical to the Euclidean decomposition of \citet{dorst_de_keninck2}, differing only in that the latter replaces multiplication by \(\ezero\) with multiplication by \(\mathcal I=\ezero e_1e_2e_3\)\added{; it also coincides with the decomposition into bulk and weight of \citet[\S2.8.3]{lengyel1}}.\end{rexample}

\begin{corollary}\label{theorem:complement_times_ezero_equals_ideal}For every \(\mvecX\in\Cl(\vsV)\), we have that \(\Cl(\pi_\vsW)(\mvecX)\ezero=\mvecX\ezero\). That is, the sets \(\Cl(\vsV)\ezero\) and \(\Cl(\vsW)\ezero\) are equal.\end{corollary}

We can thus define the projection \begin{align*}
    \derD_\vsW:\Cl(\vsV)&\to\Cl(\vsV)\ezero\\
    \mvecX&\mapsto\mvecX-\Cl(\pi_\vsW)(\mvecX),
\end{align*} so that \(\mvecX=\Cl(\pi_\vsW)(\mvecX)+\derD_\vsW(\mvecX)\). The idempotence of \(\derD_\vsW\) follows from the fact that \(\derD_\vsW=\id-\Cl(\pi_\vsW)\), as does the fact that \(\derD_\vsW(\vsV)\subseteq\vsV\), but it is not an algebra homomorphism.

\begin{lemma}\label{theorem:projection_is_derivation}\(\derD_\vsW\) is a Clifford algebra derivation and fixes \(\ezero\).\end{lemma}
\begin{proof}
    First, \(\Cl(\pi_\vsW)(\ezero)=(\omega_\vsW\circ\pi)(\ezero)=0\), so \(\derD_\vsW(\ezero)=\ezero-\Cl(\pi_\vsW)(\ezero)=\ezero\). Now let \(k\) be a nonnegative integer and \(\vecV_0,\dotsc\vecV_k\in\vsV\). For each \(0\leq i\leq k\), by \Cref{theorem:playfair_decomposition}, there exists \(\lambda_i\in\F\) such that \(\derD_\vsW(\vecV_i)=\vecV_i-\pi_\vsW(\vecV_i)=\lambda_i\ezero\). Thus we have that \begin{align*}
        \derD_\vsW(\vecV_0\dotsm\vecV_k)&=\vecV_0\dotsm\vecV_k-\Cl(\pi_\vsW)(\vecV_0\dotsm\vecV_k)\\
        &=\vecV_0\dotsm\vecV_k-(\vecV_0-\lambda_0\ezero)\dotsm(\vecV_k-\lambda_k\ezero)\\
        &=\sum_{i=0}^k\lambda_i\vecV_0\dotsm\vecV_{i-1}\gradeInv(\vecV_{i+1}\dotsm\vecV_k)\ezero,
    \end{align*} and hence that \begin{align*}
        \derD_\vsW(\vecV_0\dotsm\vecV_k)&=\sum_{i=0}^k\lambda_i\vecV_0\dotsm\vecV_{i-1}\gradeInv(\vecV_{i+1}\dotsm\vecV_k)\ezero\\
        &=\lambda_0\gradeInv(\vecV_1\dotsm\vecV_k)\ezero+\vecV_0\sum_{i=1}^k\lambda_i\vecV_1\dotsm\vecV_{i-1}\gradeInv(\vecV_{i+1}\dotsm\vecV_k)\ezero\\
        &=\vecV_0\paren*{\sum_{i=1}^k\lambda_i\vecV_1\dotsm\vecV_{i-1}\gradeInv(\vecV_{i+1}\dotsm\vecV_k)\ezero}+(\lambda_0\ezero)\vecV_1\dotsm\vecV_k\\
        &=\vecV_0\derD_\vsW(\vecV_1\dotsm\vecV_k)+\derD_\vsW(\vecV_0)\vecV_1\dotsm\vecV_k,
    \end{align*} so \(\derD_\vsW\) is a derivation on finite products of elements of \(\vsV\). As \(\derD_\vsW\) is trivially linear and since \(\Cl(\vsV)\) is finitely spanned by finite products of elements of \(\vsV\), we have that \(\derD_\vsW\) is a derivation on all of \(\Cl(\vsV)\). Finally, as \(\derD_\vsW(\vecV)\in\F\ezero\subseteq\vsV\) for all \(\vecV\in\vsV\), we have that \(\derD_\vsW\) is a Clifford algebra derivation.
\end{proof}

\begin{lemma}\label{theorem:quotient_algebra_isomorphic_algebra_of_quotient}The quotient algebra \(\Cl(\vsV)/\Cl(\vsV)\ezero\) is canonically isomorphic to the Clifford algebra of the quotient space \(\Cl(\vsV/\F\ezero)\).\end{lemma}
\begin{proof}
    For any \(\mvecX\ezero\in\Cl(\vsV)\ezero\), we see that \(\Cl(\pi)(\mvecX\ezero)=\Cl(\pi)(\mvecX)\pi(\ezero)=0\), so \(\Cl(\vsV)\ezero\subseteq\ker\Cl(\pi)\). By \Cref{theorem:clifford_algebra_playfair_decomposition}, for any \(\mvecX\in\ker\Cl(\pi)\) there exists \(\mvecY\in\Cl(\vsW)\) such that \(0=\Cl(\omega_\vsW\circ\pi)(\mvecX)=\Cl(\pi_\vsW)(\mvecX)=\mvecX-\mvecY\ezero\), so \(\mvecX=\mvecY\ezero\in\Cl(\vsV)\ezero\) and thus \(\ker\Cl(\pi)=\Cl(\vsV)\ezero\). As \(\pi\) is a surjection, so is \(\Cl(\pi)\), hence \(\im\Cl(\pi)=\Cl(\vsV/\F\ezero)\). Therefore \(\Cl(\pi)\) is an isomorphism from \(\Cl(\vsV)/\Cl(\vsV)\ezero\) to \(\Cl(\vsV/\F\ezero)\) by the First Isomorphism Theorem.
\end{proof}

As we have a canonical isomorphism between the quotient algebra of the Clifford algebra \(\Cl(\vsV)/\Cl(\vsV)\ezero\) and the Clifford algebra of the quotient space \(\Cl(\vsV/\F\ezero)\), we shall abuse notation by identifying the two algebras, as well as identifying their elements.

\begin{corollary}\label{theorem:proper_nontrivial_twosided_ideal}The ideal \(\Cl(\vsV)\ezero\) is proper, nontrivial, and two-sided.\end{corollary}

The element \(\ezero\) (or any other nonzero degenerate element of \(\vsV\)) is remarkable in satisfying \Cref{theorem:proper_nontrivial_twosided_ideal}. Indeed, consider a nonnull element \(\vecV\in\vsV\); as \(\vecV\) is invertible in \(\Cl(\vsV)\), the ideal \(\Cl(\vsV)\vecV\) equals \(\Cl(\vsV)\). The zero element \(0\in\vsV\), on the other hand, generates the zero (trivial) ideal. A nondegenerate null element generates proper nontrivial left and right ideals, but the two are not equal.

\begin{lemma}\label{theorem:derivation_is_projection}Let \(\derD:\Cl(\vsV)\to\Cl(\vsV)\ezero\) be a Clifford algebra derivation such that \(\derD(\ezero)=\ezero\). Then \(\derD=\derD_\vsW\) for some complement \(\vsW\) of \(\F\ezero\) in \(\vsV\).\end{lemma}
\begin{proof}
    We claim that the map \begin{align*}
        \Omega_\derD:\Cl(\vsV/\F\ezero)&\to\Cl(\vsV)\\
        \eqc\mvecX&\mapsto\mvecX-\derD(\mvecX);
    \end{align*} is well-defined. Indeed, for any \(\mvecX,\mvecY\in\Cl(\vsV)\), if \(\eqc\mvecX=\eqc\mvecY\) then there exists \(\mvecA\in\Cl(\vsV)\) such that \(\mvecX-\mvecY=\mvecA\ezero\), and \(\mvecB\in\Cl(\vsV)\) such that \(\derD(\mvecA)=\mvecB\ezero\). Then \begin{align*}
        \Omega_\derD(\eqc\mvecX)-\Omega_\derD(\eqc\mvecY)&=(\mvecX-\derD(\mvecX))-(\mvecY-\derD(\mvecY))\\
        &=\mvecA\ezero-\derD(\mvecA\ezero)\\
        &=\mvecA\ezero-\mvecA\derD(\ezero)-\derD(\mvecA)\ezero\\
        &=\mvecA\ezero-\mvecA\ezero-\mvecB\ezero^2\\
        &=0,
    \end{align*} proving the claim.

    Now note that for any \(\mvecX,\mvecY\in\Cl(\vsV)\), we have that \(\derD(\mvecX)=\mvecA\ezero\) and \(\derD(\mvecY)=\mvecB\ezero\) for some \(\mvecA,\mvecB\in\Cl(\vsV)\); by \Cref{theorem:commuting_ezero} we then have \begin{equation}\label{equation:derivation_squared_zero}
        \derD(\mvecX)\derD(\mvecY)=(\mvecA\ezero)(\mvecB\ezero)=\mvecA\gradeInv(\mvecB)\ezero^2=0.
    \end{equation} Let \(k\) be a nonnegative integer and \(\vecV_0,\dotsc,\vecV_k\in\vsV\). Then \begin{align*}
        \Omega_\derD(\eqc{\vecV_0\dotsm\vecV_k})={}&\vecV_0\dotsm\vecV_k-\vecV_0\derD(\vecV_1\dotsm\vecV_k)\\
        &-\derD(\vecV_0)\vecV_1\dotsm\vecV_k+\derD(\vecV_0)\derD(\vecV_1\dotsm\vecV_k)\\
        ={}&(\vecV_0-\derD(\vecV_0))(\vecV_1\dotsm\vecV_k-\derD(\vecV_1\dotsm\vecV_k))\\
        ={}&\Omega_\derD(\eqc{\vecV_0})\Omega_\derD(\eqc{\vecV_1\dotsm\vecV_k}),
    \end{align*} where we used \Cref{equation:derivation_squared_zero} to add \(\derD(\vecV_0)\derD(\vecV_1\dotsm\vecV_k)=0\). Hence, as the algebra \(\Cl(\vsV/\F\ezero)\) is finitely spanned by finite products of elements of \(\vsV/\F\ezero\), we have that \(\Omega_\derD\) is a unital algebra homomorphism; as \(\derD\) is a Clifford algebra derivation, \(\Omega_\derD=\id-\derD\) also fixes \(\vsV\) setwise and so is a Clifford algebra homomorphism. Thus \(\Omega_\derD\) restricts to an isometry \(\omega_\derD\coloneqq\Omega_\derD|_{\vsV/\F\ezero}:\vsV/\F\ezero\to\vsV\). Composing with \(\pi\) we get, for any \(\eqc\vecV\in\vsV/\F\ezero\), \begin{align*}
        (\pi\circ\omega_\derD)(\eqc\vecV)=\pi(\vecV-\derD(\vecV))=\eqc\vecV,
    \end{align*} noting that \(\derD(\vecV)\in\vsV\cap\Cl(\vsV)\ezero=\F\ezero\). Hence \(\omega_\derD\) is a section of \(\pi\) and so \(\vsW\coloneqq\omega_\derD(\vsV)\) is a complement of \(\F\ezero\) in \(\vsV\); by \Cref{theorem:orthogonal_complements} then \(\omega_\derD=\omega_\vsW\). Finally then, for any \(\mvecX\in\Cl(\vsV)\), \begin{align*}
        \derD_\vsW(\mvecX)-\derD(\mvecX)&=(\mvecX-\Cl(\pi_\vsW)(\mvecX))-(\mvecX-\Omega_\derD(\eqc\mvecX))\\
        &=\Omega_\derD(\eqc\mvecX)-\Cl(\omega_\vsW)(\eqc\mvecX)\\
        &=\Cl(\omega_\derD)(\eqc\mvecX)-\Cl(\omega_\vsW)(\eqc\mvecX)\\
        &=0,
    \end{align*} so \(\derD=\derD_\vsW\).
\end{proof}

\begin{theorem}\label{theorem:complements_and_derivations}The complements of \(\F\ezero\) in \(\vsV\) are in one-to-one correspondence with the Clifford algebra derivations from \(\Cl(\vsV)\) to \(\Cl(\vsV)\ezero\) fixing \(\ezero\).\end{theorem}
\begin{proof}
    Follows from \Cref{theorem:projection_is_derivation,theorem:derivation_is_projection}.
\end{proof}

\begin{rexample}
    The complements \(\vsV_\ptP\) of \(\R\ezero\) are in one-to-one correspondence with points \(\ptP\) in space. \Cref{theorem:complements_and_derivations} gives us another correspondence, with a natural class of derivations.
\end{rexample}

\begin{lemma}\label{theorem:right_multiplication_isomorphism}Let \(\lambda\in\F^\times\). Then right-multiplication by \(\lambda\ezero\) is a linear isomorphism from \(\Cl(\vsW)\) to the ideal \(\Cl(\vsV)\ezero\).\end{lemma}
\begin{proof}
    Let \(\rho\) denote right-multiplication by \(\lambda\ezero\); i.e. \(\rho:\mvecX\mapsto\lambda\mvecX\ezero\). Immediately, \(\rho\) is linear by the bilinearity of the Clifford algebra product. Now let \(\mvecX\in\Cl(\vsW)\cap\ker\rho\). Then \(\lambda\mvecX\wedge\ezero=\lambda\mvecX\ezero=0\), and as \(\vecW\) and \(\ezero\) are linearly independent for every \(\vecW\in\vsW\), it must be that \(\mvecX=0\); hence \(\rho\) is injective on \(\Cl(\vsW)\). Finally, \(\rho\) is surjective on \(\Cl(\vsW)\) by \Cref{theorem:complement_times_ezero_equals_ideal} and thus is a linear isomorphism.
\end{proof}

\begin{theorem}\label{theorem:fundamental_theorem_of_aga}The following Clifford algebras are canonically isomorphic:\begin{enumerate}
    \item \(\Cl(\vsV)/\Cl(\vsV)\ezero\), the quotient algebra,
    \item \(\Cl(\vsV/\F\ezero)\), the Clifford algebra of the quotient space, and
    \item \(\Cl(\vsW)\), the Clifford algebra of any complement of \(\F\ezero\) in \(\vsV\).
\end{enumerate}

Moreover, each \(\vecU\in\F\ezero\setminus\set0\) gives a distinct linear isomorphism \(\mvecX\mapsto\mvecX\vecU\) from the above vector spaces to the ideal \(\Cl(\vsV)\ezero=\Cl(\vsW)\ezero\).\end{theorem}
\begin{proof}
    (1) and (2) are isomorphic by \Cref{theorem:quotient_algebra_isomorphic_algebra_of_quotient}. (2) and (3) are isomorphic via the Clifford extension of the canonical section \(\omega_\vsW\). \Cref{theorem:right_multiplication_isomorphism} establishes the family of linear isomorphisms to \(\Cl(\vsV)\ezero\), which by \Cref{theorem:complement_times_ezero_equals_ideal} is equal to \(\Cl(\vsW)\ezero\).
\end{proof}

\begin{rexample}
    In the context of Euclidean PGA, the \added{quotient} Clifford algebra \(\Cl(\vsV/\R\ezero)\) (likewise \(\Cl(\vsV)/\Cl(\vsV)\ezero\)) describes the objects in and transformations of the space of parallel classes of planes in Euclidean space. Along with the parallel classes of planes themselves, we see two-parameter bundles of parallel lines, and one ``parallel class'' comprising all points in space. We also see reflections and rotations of space; note that the action of, say, a rotation on a parallel class of planes does not depend on the axis of rotation, only on the angle and the parallel class to which the axis belongs, so a ``parallel class'' of rotations carries only this information. As we see from \Cref{theorem:fundamental_theorem_of_aga}, these objects and transformations correspond (canonically!) one-to-one with the objects through and transformations fixing a given point \(\ptP\): the elements of \(\Cl(\vsV_\ptP)\). The second statement in \Cref{theorem:fundamental_theorem_of_aga} says that each of these parallel classes also corresponds one-to-one with an element at infinity in the ideal \(\Cl(\vsV)\ezero\); for instance, a parallel bundle of lines meets infinity at a single ideal point. The correspondence is canonical up to a choice of scalar multiple of \(\ezero\).
\end{rexample}

A \definedTerm{square-zero ideal} is an ideal \(I\) such that \(I^2=\set{ij\mid i,j\in I}=\set0\). By \Cref{theorem:commuting_ezero}, the ideal \(\Cl(\vsV)\ezero\) is a square-zero ideal. As \(\Cl(\vsV)\ezero\) is the kernel of \(\Cl(\pi)\), this makes \(\Cl(\vsV)\) a \definedTerm{square-zero extension} of \(\Cl(\vsW)\). The most well-known square-zero extension is the \definedTerm{trivial extension} of a ring \(R\) by an \(R\)-bimodule \(M\), which is the ring \(R\ltimes M\) given by the direct sum \(R\oplus M\) and the product \((r_1,m_1)(r_2,m_2)\coloneqq(r_1r_2,r_1m_2+m_1r_2)\). This is not quite what we want, as we will see; we turn instead to the \definedTerm{twisted trivial extension} as defined by \citet{guo1}, in which the module \(M\) is replaced by its twist \(M^\sigma\) for a ring automorphism \(\sigma\) of \(R\). The module \(M^\sigma\) is the same left \(R\)-module as \(M\), but its right multiplication is given by \(m\cdot r\coloneqq m\sigma(r)\). The twisted trivial extension \(R\ltimes_\sigma M\coloneqq R\ltimes M^\sigma\) thus has product \((r_1,m_1)(r_2,m_2)=(r_1r_2,r_1m_2+m_1\sigma(r_2))\).

\begin{theorem}\label{theorem:clifford_algebra_decomposition}The degenerate Clifford algebra \(\Cl(\vsV)\) is isomorphic to the twisted trivial extension \(\Cl(\vsW)\ltimes_\gradeInv\Cl(\vsW)\), and hence also to \(\Cl(\vsV/\F\ezero)\ltimes_\gradeInv\Cl(\vsV/\F\ezero)\).\end{theorem}
\begin{proof}
    By \Cref{theorem:clifford_algebra_playfair_decomposition}, we have that \(\Cl(\vsV)=\Cl(\vsW)\oplus\Cl(\vsW)\ezero\); by \Cref{theorem:fundamental_theorem_of_aga} this is linearly isomorphic to \(\Cl(\vsW)\oplus\Cl(\vsW)\) via \(\phi:\mvecX+\mvecY\ezero\mapsto(\mvecX,\mvecY)\). Also, \[\begin{split}
        \phi((\mvecX_1+\mvecY_1\ezero)(\mvecX_2+\mvecY_2\ezero))&=\phi(\mvecX_1\mvecX_2+\mvecX_1\mvecY_2\ezero+\mvecY_1\ezero\mvecX_2+\mvecY_1\ezero\mvecY_2\ezero)\\
        &=\phi(\mvecX_1\mvecX_2+(\mvecX_1\mvecY_2+\mvecY_1\gradeInv(\mvecX_2))\ezero)\\
        &=(\mvecX_1\mvecX_2,\mvecX_1\mvecY_2+\mvecY_1\gradeInv(\mvecX_2))\\
        &=(\mvecX_1,\mvecY_1)(\mvecX_2,\mvecY_2)\\
        &=\phi(\mvecX_1+\mvecY_1\ezero)\phi(\mvecX_2+\mvecY_2\ezero),
    \end{split}\] so \(\phi\) is an algebra isomorphism from \(\Cl(\vsV)\) to \(\Cl(\vsW)\ltimes_\gradeInv\Cl(\vsW)\). As earlier, the Clifford extension of the canonical section \(\omega_\vsW\) gives the isomorphism to \(\Cl(\vsV/\F\ezero)\ltimes_\gradeInv\Cl(\vsV/\F\ezero)\).
\end{proof}

\added{We may also think about decomposing the even subalgebra \(\Cl^+(\vsV)\) and the odd subspace \(\Cl^-(\vsV)\) of \(\Cl(\vsV)\). It is fairly clear from \Cref{theorem:clifford_algebra_decomposition} that \(\Cl^+(\vsV)=\Cl^+(\vsW)\ltimes\Cl^-(\vsW)\ezero\cong\Cl^+(\vsW)\ltimes_\gradeInv\Cl^-(\vsW)\). But as \(\gradeInv\) is the identity map when restricted to \(\Cl^+(\vsW)\), this is just \[
    \Cl^+(\vsV)\cong\Cl^+(\vsW)\ltimes\Cl^-(\vsW)\cong\Cl^+(\vsV/\F\ezero)\ltimes\Cl^-(\vsV/\F\ezero),
\] the ``untwisted'' trivial extension of the even subalgebra by the odd subspace. Similarly we get that the odd subspace \(\Cl^-(\vsV)=\Cl^-(\vsW)\oplus\Cl^+(\vsW)\ezero\), or \[
    \Cl^-(\vsV)\cong\Cl^+(\vsW)\oplus\Cl^-(\vsW)\cong\Cl^+(\vsV/\F\ezero)\oplus\Cl^-(\vsV/\F\ezero),
\] swapping the order of direct summands for the sake of aesthetics.}

\section{The group of units and Lie algebra of bivectors}\label{section:group_of_units}
We denote the \definedTerm{group of units} of a unital algebra \(\algA\) by \(\algA^\times\), and for a homomorphism \(\phi:\algA\to\algB\), denote the \definedTerm{restriction to units} \(\phi^\times\coloneqq\phi|_{\algA^\times}:\algA^\times\to\algB^\times\). This defines the functor \(-^\times:\Alg_\F\to\Grp\) (see \citet[p. 50, Ex. 8]{cohn1}). The proof of the following lemma is adapted from a similar one by \citet[Theorem 3.7]{anderson_winders1} about the units of a trivial extension of a commutative ring.

\begin{lemma}\label{theorem:units_of_twisted_trivial_extension}Let \(R\) be a ring with an automorphism \(\sigma\), and \(M\) an \(R\)-bimodule. The group of units \((R\ltimes_\sigma M)^\times\) of the twisted trivial extension \(R\ltimes_\sigma M\) is \(R_0^\times\ltimes(1+M_0)\cong R^\times\ltimes_\tau M\), where \(R_0=\set{(r,0)\mid r\in R}\), \(1+M_0=\set{(1,m)\mid m\in M}\), and \(\tau:R^\times\to\Aut_R(M)\) is given by \(\tau(r):m\mapsto rm\sigma(r^{-1})\).\end{lemma}
\begin{proof}
    It is straightforward to show that \({R_0}^\times\) is a subgroup of \((R\ltimes_\sigma M)^\times\) isomorphic to \(R^\times\), that \(1+M_0\) is an abelian subgroup of \((R\ltimes_\sigma M)^\times\) isomorphic to \(M\), and that the two subgroups intersect trivially. Now suppose that \((r,m)\in(R\ltimes_\sigma M)^\times\); then there exists \((s,n)\in R\ltimes_\sigma M\) such that \((r,m)(s,n)=(rs,rn+m\sigma(s))=(1,0)\). Hence \(rs=1\), so \((r,0)(s,0)=(1,0)\) and hence \((r,0)\) is a unit. Thus \((r,0)(1,sm)=(r,m)\) and so \((R\ltimes_\sigma M)^\times={R_0}^\times\ltimes(1+M_0)\), where the action of \({R_0}^\times\) on \(1+M_0\) is given by \[
        (1,m)\mapsto(r^{-1},0)(1,m)(r,0)=(1,r^{-1}m\sigma(r))=(1,\tau(r^{-1})(m)).
    \] Therefore \((R\ltimes_\sigma M)^\times\cong R^\times\ltimes_\tau M\).
\end{proof}

Recall that we assume a degenerate quadratic space \(\vsV\) with \(\ezero\in\rad\vsV\) and \(\vsW\) some complement of \(\F\ezero\) in \(\vsV\).

\begin{corollary}\label{theorem:group_of_units_decomposition} The group of units of \(\Cl(\vsV)\) is \(\Cl(\vsW)^\times\ltimes(1+\Cl(\vsW)e_0)\), which is isomorphic to \(\Cl(\vsV/\F\ezero)^\times\ltimes_\tau\Cl(\vsV/\F\ezero)\) and \(\Cl(\vsW)^\times\ltimes_\tau\Cl(\vsW)\).\end{corollary}

This establishes the decomposition of the group of units of a degenerate Clifford algebra. Of course, the interesting groups in a Clifford algebra are proper subgroups of the group of units; things get much more complicated here. The authors refer the reader to \citet{ablamowicz2} and \citet{filimoshina_shirokov2} for more detailed treatments of these subgroups. At every level, we see extensions of the form \(G\ltimes(1+\vsM\ezero)\), where \(G\) is some subgroup of \(\Cl(\vsW)^\times\) and \(\vsM\) is some subspace of \(\Cl(\vsW)\).

Also interesting is the Lie algebra of grade-2 elements or \definedTerm{bivectors} (see \citet[Ch. 8]{dorst_fontijne_mann1} for a gentle introduction). For any bivector \(\bvecB\in\Cl^2(\vsV)\) and any \(\mvecX\in\Cl^k(\vsV)\), the \definedTerm{commutator} \[
    \bvecB\times\mvecX\coloneqq\frac12(\bvecB\mvecX-\bvecB\mvecX)
\] is in \(\Cl^k(\vsV)\). In particular, \(\Cl^2(\vsV)\) is closed, and hence a Lie algebra, under this commutator. It is straightforward to show, by writing each element \(\bvecB\in\Cl^2(\vsV)\) in the form \[
    \bvecB=\vecU_1\vecV_1+\vecU_2\vecV_2+\dotsb+\vecU_k\vecV_k,
\] where \(\sbfB(\vecU_i,\vecV_i)=0\) for each \(1\leq i\leq k\), that a Clifford algebra homomorphism \(\phi:\Cl(\vsV)\to\Cl(\vsV')\) restricts to a Lie algebra homomorphism \(\phi^2:\Cl^2(\vsV)\to\Cl^2(\vsV')\), and so \(-^2\) is a functor from \(\Cl(\Quad_\F)\) to \(\LieAlg_\F\), the category of Lie algebras and Lie algebra homomorphisms. \Citet[Lemma 5.7, p. 182]{gracia-bondia_varilly_figueroa1} show that \(\Cl^2(\vsV)\cong\so(\vsV)\) in the case when \(\vsV\) is real and nondegenerate. What happens in the degenerate case?

\begin{lemma}Let \(\bvecB\in\Cl^2(\vsV)\). Then \(\derD_\vsW(\bvecB)=\vecW\ezero\) for some \(\vecW\in\vsW\).\end{lemma}
\begin{proof}
    Let \(\vecV_1,\vecV_2\in\vsV\) and \(\sbfB(\vecV_1,\vecV_2)=0\). Then by \Cref{theorem:playfair_decomposition}, there exist \(\vecW_1,\vecW_2\in\vsW\) and \(\lambda_1,\lambda_2\in\F\) such that \(\vecV_1=\vecW_1+\lambda_1\ezero\) and \(\vecV_2=\vecW_2+\lambda_2\ezero\). Hence for the \definedTerm{simple} bivector \(\vecV_1\vecV_2\), we have \begin{align*}
        \derD_\vsW(\vecV_1\vecV_2)&=\derD_\vsW((\vecW_1+\lambda_1\ezero)(\vecW_2+\lambda_2\ezero))\\
        &=\derD_\vsW(\vecW_1\vecW_2+(\lambda_2\vecW_1\replaced{-}{+}\lambda_1\vecW_2)\ezero)\\
        &=(\lambda_2\vecW_1\replaced{-}{+}\lambda_1\vecW_2)\ezero\in\vsW\ezero.
    \end{align*} Now as \(\bvecB\) may be written as a finite sum of simple bivectors, we see that \(\derD_\vsW(\bvecB)\) is a finite sum of elements of \(\vsW\ezero\) and hence is in \(\vsW\ezero\); that is, there exists \(\vecW\in\vsW\) such that \(\derD_\vsW(\bvecB)=\vecW\ezero\). 
\end{proof}

We see that the idempotent and surjective Clifford algebra homomorphism \(\Cl(\pi_\vsW)\) (resp. derivation \(\derD_\vsW\)) restricts to an idempotent and surjective Lie algebra homomorphism (resp. derivation), so the bivector Lie algebra \(\Cl^2(\vsV)\) splits into the two subspaces \(\Cl(\pi_\vsW)(\Cl^2(\vsV))=\Cl^2(\vsW)\) and \(\derD_\vsW(\Cl^2(\vsV))=\ker\Cl(\pi_\vsW)=\vsW\ezero\).

\begin{theorem}\label{theorem:lie_algebra_decomposition}The Lie algebra of bivectors \(\Cl^2(\vsV)\) is equal to the semidirect sum \(\Cl^2(\vsW)\ltimes\vsW\ezero\), which is isomorphic to \(\Cl^2(\vsW)\ltimes_\tau\vsW\) where \(\tau:\Cl^2(\vsW)\to\Der(\vsW)\) is given by \(\tau(\bvecB):\vecW\mapsto\bvecB\times\vecW\), and to \(\Cl(\vsV/\F\ezero)\ltimes_{\tau'}\vsV/\F\ezero\) where \(\tau'\) is defined similarly.\end{theorem}
\begin{proof}
    First note that \(\vsW\ezero\) is the kernel of \(\Cl^2(\pi_\vsW)\) and so an ideal of \(\Cl^2(\vsV)\), and that \(\Cl^2(\vsW)\) is its image and hence a subalgebra. Thus we have a split extension of Lie algebras and so \(\Cl^2(\vsV)=\Cl^2(\vsW)\ltimes\vsW\ezero\), where the action of \(\Cl^2(\vsW)\) on \(\vsW\ezero\) is given by \begin{align*}
        \vecW\ezero&\mapsto\bvecB\times\vecW\ezero=\frac12(\bvecB\vecW\ezero-\vecW\ezero\bvecB)=(\bvecB\times\vecW)\ezero=\tau(\bvecB)(\vecW)\ezero
    \end{align*} because \(\gradeInv(\bvecB)=\bvecB\). Now \(\vsW\ezero\) is an abelian ideal as, for any \(\vecW_1\ezero,\vecW_2\ezero\in\vsW\ezero\), we have that \(\vecW_1\ezero\times\vecW_2\ezero=0\); hence \(\vecW\ezero\mapsto\vecW\) is an isomorphism from \(\vsW\ezero\) to the abelian Lie algebra \(\vsW\). Also because \(\vsW\) is an abelian Lie algebra, the linear map \(\tau(\bvecB)\) is trivially a derivation on \(\vsW\) for all \(\bvecB\in\Cl^2(\vsW)\). Let \(\bvecB_1,\bvecB_2\in\Cl^2(\vsW)\); then for all \(\vecW\in\vsW\), \begin{align*}
        \tau(\bvecB_1\times\bvecB_2)(\vecW)={}&(\bvecB_1\times\bvecB_2)\times\vecW\\
        ={}&\frac14((\bvecB_1\bvecB_2-\bvecB_2\bvecB_1)\vecW-\vecW(\bvecB_1\bvecB_2-\bvecB_2\bvecB_1))\\
        ={}&\frac14(\bvecB_1(\bvecB_2\vecW-\vecW\bvecB_2)-(\bvecB_2\vecW-\vecW\bvecB_2)\bvecB_1\\
        &-\bvecB_2(\bvecB_1\vecW-\vecW\bvecB_1)+(\bvecB_1\vecW-\vecW\bvecB_1)\bvecB_2)\\
        ={}&\bvecB_1\times(\bvecB_2\times\vecW)-\bvecB_2\times(\bvecB_1\times\vecW)\\
        ={}&(\tau(\bvecB_1)\circ\tau(\bvecB_2)-\tau(\bvecB_2)\circ\tau(\bvecB_1))(\vecW),
    \end{align*} so \(\tau\) is indeed a Lie algebra homomorphism. \Cref{theorem:fundamental_theorem_of_aga} establishes the required isomorphisms for the remainder of the theorem.
\end{proof}

\begin{rexample}The significance of \Cref{theorem:lie_algebra_decomposition} is quite clear in our running example, where we have the known isomorphisms \(\Cl^2(\vsV/\R\ezero)\cong\so(3)\) and \(\vsV/\R\ezero\cong\R^3\). Hence we get \begin{align*}
    \Cl^2(\vsV)\cong\so(3)\ltimes\R^3=\se(3),
\end{align*} the Lie algebra of infinitesimal motions of Euclidean space.\end{rexample}

\section{Conclusion}
Classical treatments of degenerate quadratic spaces and their Clifford algebras have focused on splitting into a nondegenerate quotient space and the radical. In this article, inspired by the unique structure of the affine geometry found in Euclidean PGA, we describe a new ``one dimension at a time" approach, yielding trivial decompositions of the corresponding Clifford algebra (\Cref{theorem:clifford_algebra_decomposition}), its group of units (\Cref{theorem:group_of_units_decomposition}), and its Lie algebra of bivectors (\Cref{theorem:lie_algebra_decomposition}). We also find, in \Cref{theorem:complements_and_derivations}, a surprising correspondence between a class of derivations of the Clifford algebra and points in the affine space described.

We also show that, in the case of the degenerate quadratic space used by Euclidean PGA, the quotient space has a natural interpretation as the space of parallel classes of planes. This interpretation, combined with the results above, helps provide mathematical background for the Euclidean split described by \citet{dorst_de_keninck2}, and may lead to new methods. The isomorphisms of \Cref{theorem:fundamental_theorem_of_aga} appear particularly likely to bear fruit; consider, for example, that lines at infinity correspond one-to-one with parallel classes of planes in space, and that the norm \(\vecV^2=(\vecV+\ezero)^2\) is well-defined on parallel classes of planes. This interpretation should work for any metric affine geometry -- wherever we find unique parallels -- in any dimension and with any degenerate signature; the authors see this avenue as a kind of \definedTerm{affine geometric algebra}, standing next to PGA among modern geometric algebra's various flavours.

\section*{Statements and Declarations}
This research was conducted during the latter author's pursuit of a Bachelor of Science (Hons) degree at the \replaced{University of Western Australia}{same university}, funded partially by the Australian Government. The latter author would like to thank his fellow Honours students for many interesting and productive conversations about mathematics, and Hamish Todd for his confidence and support.

\bibliographystyle{jabbrv_abbrvnat}
\bibliography{bibliography}

\end{document}